\documentclass[final]{siamart190516}

\usepackage[utf8]{inputenc}
\usepackage{amsmath}
\usepackage{mathtools}
\usepackage{mathrsfs}
\usepackage{amssymb}
\usepackage{graphicx}
\usepackage{subcaption}
\usepackage{chngcntr}

\usepackage{epstopdf}
\ifpdf
  \DeclareGraphicsExtensions{.eps,.pdf,.png,.jpg}
\else
  \DeclareGraphicsExtensions{.eps}
\fi

\newsiamremark{remark}{Remark}
\newsiamremark{hypothesis}{Hypothesis}
\crefname{hypothesis}{Hypothesis}{Hypotheses}

\headers{Homotopies and transcendental extensions}{Wojciech Duli\'nski}

\title{Homotopies and transcendental extensions in colouring problems}

\author{Wojciech Duli\'nski 
  (\email{w.dulinski@uw.edu.pl})
  \\University of Warsaw, Poland.}

\usepackage{amsopn}

\ifpdf
\hypersetup{
  pdftitle={Homotopies and transcendental extensions in problems concerning colouring of vertices},
  pdfauthor={Wojciech Duliński}
}
\fi

\begin{document}

\maketitle



\textit{\textbf{Key words -}} colorings, Sperner, polytopes, triangulations
\section{Introduction}  
The study of vertex colorings in triangulated polytopes is a classical topic in combinatorial topology. A foundational result in this field is Sperner's lemma, which, while equivalent to Brouwer's fixed point theorem, can be expressed in purely combinatorial terms. In this paper, we advance existing theorems using the oriented volume method of McLennan and Tourky \cite{sperner1}—a framework that replaces combinatorial arguments with polynomial homotopy techniques. For deeper results, we further refine this method by integrating geometric realizations with algebraically independent coordinates, yielding the following results:

\begin{theorem}\label{gra-y}  
Let $\Delta^n$ be an $n$-dimensional simplex with a given triangulation. Let $V$ be the set of vertices of this triangulation. If $V$ is colored with $n$ colors, there exists a connected monochromatic subgraph that intersects every codimension 1 face of $\Delta^n$.  
\end{theorem}  

This result follows directly from the oriented volume method and serves as a motivation and concise illustration of its utility. While being simple, it admits a nice interpretation as a statement about the impossibility of a draw in the generalized Y game, extending the 2-dimensional case established by Prytuła \cite{prytula} and linking the theory presented here to classical result of Gale on the game of Hex \cite{gale}.

\begin{theorem}\label{prod-sympleks}
Let $T$ be a triangulation of $\Delta^n \times \Delta^m$ with a Sperner labeling by $(n+1)(m+1)$ colors, where each extreme point of $\Delta^n \times \Delta^m$ has a distinct label. For any tuple $(k_0, \ldots, k_m)$ with $k_0 + \ldots + k_m = n+m+1$, there exists an odd number of simplices $\sigma$ in $T$ of type $(k_0, \ldots, k_m)$ (i.e., simplices with exactly $k_i$ vertices in the fiber over the $i$-th vertex of $\Delta^m$ under the projection $\Delta^n \times \Delta^m \to \Delta^m$). Moreover, the count of such simplices $\sigma$ sharing the orientation of $\sigma'$ exceeds those with the opposite orientation by exactly one.  
\end{theorem}  

This theorem generalizes the multilabeled Sperner lemmas of Meunier and Su \cite{multi} in two key ways: it applies to arbitrary triangulations of simplex products (rather than restrictive subdivisions) and unifies their dual formulations through coordinate symmetry in $\Delta^n \times \Delta^m$, at the same time giving additional information about the parity of simplices.

\begin{theorem}[Multilabeled Ky Fan's Lemma]\label{mult-kyfan}  
Let $P$ be a barycentric derived subdivision of the octahedral subdivision of the $n$-disk $\mathbb{B}^n$, with triangulation $T$ that restricts to a centrally symmetric triangulation on $\partial P$. Let $\pi: P \times \Delta^{m-1} \to \Delta^{m-1}$ be the natural projection, and let $\overline{T}$ be a labeling of $P \times \Delta^{m-1}$ such that the restriction to each $\pi$-fiber over a vertex of $\Delta^{m-1}$ is a Fan labeling of $P$. Then there exists an odd number of alternating $n$-simplices of any type $(k_1, \ldots, k_m)$ in $T$; these simplices become alternating $(k_i-1)$-simplices when restricted to $\pi^{-1}(i)$.  
\end{theorem}  

This final result synthesizes the oriented volume method with algebraic independence in coordinates to address the interplay of multiple Fan labelings across a product space. Its proof, situated at the culmination of our technical development, demonstrates the necessity of this hybrid approach for high-dimensional generalizations.  

In addition to these theorems, we provide new proofs of classical results such as the Atanassov conjecture and discuss the boundaries of our framework through counterexamples to plausible extensions. Collectively, this work establishes a hierarchy of techniques: from the purely combinatorial (Theorem \ref{gra-y}) to the algebraically enriched (Theorems \ref{prod-sympleks} and \ref{mult-kyfan}), illustrating how geometric-algebraic refinements expand the scope of fixed-point principles.  

\section{Basic notions}
By an \emph{$n$-simplex} $\Delta^n$, we understand a convex hull of $n+1$ affinely independent points $p_1, \ldots, p_{n+1}$ in a Euclidean space. By its \emph{$d$-face}, we mean a convex hull of $d+1$ distinct points from $\{p_1, \ldots, p_{n+1}\}$. If we don't want to specify the exact number of these points, we simply speak of a face of a simplex.

If the dimension of the ambient Euclidean space is $n$, then after a choice of orientation, the (oriented) volume of a simplex is given by the sum of determinants formed by points of $n+1$ of its $(n-1)$-faces, divided by $n!$ (this may be considered a special case of the generalized shoelace formula, see Lemma \ref{shoelace}).

By a \emph{$d$-polytopal complex}, we understand a subset $P$ of $\mathbb{R}^n$ such that:
\begin{itemize}
    \item it is a finite sum of convex hulls $C_i = \text{conv}(\{v_1^i, \ldots, v_{m_i}^i\})$ of finite subsets $\{v_1^i, \ldots, v_{m_i}^i\}$, homeomorphic to the closed ball $\mathbb{B}^r$, where $r \leq d$; such convex hulls are called \emph{$r$-cells}, while points $v_j^i$ are called \emph{vertices} of $P$;
    \item it has at least one $d$-cell;
    \item the intersection of cells $C_i, C_j$ is a cell $\text{conv}(S)$ for some $S \subset \{v_1^i, \ldots, v_{m_i}^i\} \cap \{v_1^j, \ldots, v_{m_j}^j\}$;
\end{itemize}

A polytopal complex is a \emph{polytope} if it is a convex subset of $\mathbb{R}^n$. Polytopal complexes are obviously compact and have a finite number of vertices.

We say that a point of a $d$-polytopal complex $P$ is on its \emph{boundary} $\partial P$ if it is contained in $\partial C_i$ for exactly one $d$-cell $C_i$. A standard reference on the subject of polytopes and polyhedral complexes is Ziegler \cite{ziegler}.

If $P \subset \mathbb{R}^n$ is a $d$-polytopal complex and a $d$-cell of $P$ spans the affine $d$-space $A_v$, we call $A_v$ \emph{essential} with respect to $P$. If $\lambda$ is the $d$-dimensional Lebesgue measure, we define the \emph{volume} of $P$ to be $\sum \lambda(A_v \cap P)$, where the summation is over all affine $d$-spaces essential with respect to $P$. Note that this definition does not require orientability (and is, in a sense, local).

By a \emph{triangulation} of a polytopal complex $P$, we mean a collection $T$ of simplices $\{\sigma_i\}_{i \in I}$ such that $\bigcup_{i \in I} \sigma_i = P$, for each $\sigma \in T$, every face of $\sigma$ is in $T$, and for any pair of $\sigma_i, \sigma_j$ with $i \neq j$, the set $\sigma_i \cap \sigma_j$ is a common face of $\sigma_i$ and $\sigma_j$ or is empty. By vertices of a triangulation ($T$-vertices), we mean vertices of the $\sigma_i$'s. By $T_{ess}$ we denote the collection of essential simplices of $T$, that is, simplices of the maximal dimension.

By a \emph{triangulated polytopal complex} $P$, we mean a polytopal complex in which every cell is a simplex. Note that "a triangulated polytopal complex" and "a polytopal complex with a triangulation" are distinct notions (and the distinction will be crucial), because in the latter, the extreme points of a given simplex in the triangulation need not be the vertices of the polytopal complex. Sometimes, in this case, we use the term "triangulation on the vertices of $P$" to distinguish it from a general triangulation of $P$.

By a \emph{labeling} (coloring) of a polytopal complex $P$ with $n$ colors, we mean a function $f$ from the set of vertices of $P$ to the set $\{1, \ldots, n\}$. We say that a vertex $v$ has a label (color) $f(v)$.

\section{Description of the General Proof Technique}
For the reader's convenience, we decided to include in a separate paragraph the discussion of the proof technique that we employ extensively in this article. The general idea of the oriented volume method can be inferred from the following:

\begin{lemma}
    Let $P \subset \mathbb{R}^n$ be a triangulated $n$-polytopal complex with a triangulation $T$. Fix some coloring on the set of $T$-vertices.
    Let $H: P \times [0,1] \to \mathbb{R}^n$ be a homotopy satisfying the following conditions: the vertices on the boundary of $P$ remain stationary, each vertex $v_i$ in the interior of $P$ moves in time $t \in [0,1]$ along a straight line with constant velocity towards some fixed vertex $v_i'$ and on the remaining points of $P$ the action of $H$ is determined by extending $H$ to every simplex of $T$ using barycentric coordinates (so in particular, $H$ is piecewise-linear).
    For every $\sigma \in T_{ess}$, order the vertices $\sigma_0, \ldots, \sigma_n$ of $\sigma$ in such a way that the determinant of $\big(\sigma_1-\sigma_0, \ldots, \sigma_n-\sigma_0\big)$ is positive.
    Then the volume-sum function $$[0,1] \ni t \mapsto \Sigma_{\sigma \in T_{ess}} \text{det}\big(H(\sigma_1) - H(\sigma_0), \ldots, H(\sigma_n) - H(\sigma_0)\big)$$ is constant.
\end{lemma}

\begin{proof}
    The formula for the volume-sum function is given by a polynomial in coordinates of vertices of simplices in $T_{ess}$. At time $t$, the homotopy maps coordinate $v_i$ of a vertex $v$ to some $v_i + t u_i$, where $u_i$ depends only on $v_i$ and not on $t$. Fixing the simplex $\sigma$, its image via homotopy $H_t(\sigma)$ is by construction also a simplex in the ambient space, and thus we see that the function $t \mapsto \textup{Vol}(H_t(\sigma))$ is a polynomial of one real variable. Hence, the same is true for the sum of all volumes over all simplices of maximal dimension in $T$. At $t = 0$, this is simply the volume of the triangulated polytopal complex. For small values of $t$, the vertices of the triangulation remain in the interior of $P$ and no $H_t(\sigma)$ degenerates. Hence, the corresponding images also give a triangulation of the polytopal complex, with all orientations compatible with the orientation of $P$, so this too must equal the volume of $P$ (recall that $H$ fixes the boundary). Thus, we see that the sum of $\textup{Vol}(H_t(\sigma))$'s is a polynomial function of one real variable $t \in [0,1]$ which is constant in some neighborhood of $0$. It follows that this must be a constant function.
\end{proof}

There is no need to consider the seemingly more complicated case when the dimension of $P$ is smaller than the dimension of the ambient space: if our homotopy does not take a vertex outside of $P$, then after perhaps a (linear) change of coordinates, we are reduced back to the movement in the space of the dimension of $P$.

\begin{remark}
    For the sake of simplifying the initial exposition, we made one simplifying assumption that we neither need nor want in what follows - we will allow some of the vertices on the boundary to move but only in a way that the shape of the polytopal complex is unaffected (at least for a short time) - vertices of $T$ contained in $\partial P$ are allowed to move within the smallest face to which they belong (in particular, vertices of $P$ from $\partial P$ are stationary). Vertices that are in $\text{int}\ P$ are allowed to move freely in the interior of the cell to which they belong. At time $t=0$, every point was in exactly one simplex of the triangulation (save for the boundaries of simplices, which are of measure zero). If at some time $t_0$ a pair $H_{t_0}(\sigma_1), H_{t_0}(\sigma_2)$ ($\sigma_1 \neq \sigma_2$) has an intersection of nonzero measure, then their interiors have a nonempty intersection, and then it follows that some pair of faces $\sigma_1' \subset \sigma_1, \sigma_2' \subset \sigma_2$ (not necessarily of the same dimension) that were disjoint at $t=0$ intersected at some $t_1 < t_0$ - $H_{t_1}(\sigma_1') \cap H_{t_0}(\sigma_1') \neq \emptyset$. Since functions describing the distance between such pairs are continuous in $t$, and there are finitely many of these functions, we see that for some short time interiors of all simplices were disjoint. Hence, up to a measure zero set, every point was inside at most one simplex (obviously positively oriented). Moreover, $H_t$ is surjective - it attains all of the extreme points of the polytopal complex by the assumption and for the remaining points one can use induction on the dimension of the face in which the point is contained (and the fact that cells are convex sets).
\end{remark}

We will also make use of the following fact from field theory (see Browkin \cite{fields} or Milne \cite{fields2}):

\begin{lemma}
    Suppose that $\mathbb{K} \subset \mathbb{L}$ is a field extension. Then there exists a field $\mathbb{M}$ such that $\mathbb{K} \subset \mathbb{M} \subset \mathbb{L}$, where the first extension is purely transcendental and the second is algebraic.
\end{lemma}

In particular, applying this to $\mathbb{K}$ being the field of real algebraic numbers and $\mathbb{L} = \mathbb{R}$, we obtain a subfield $\mathbb{M}$ of $\mathbb{R}$ containing $\mathbb{K}$, purely transcendental over $\mathbb{K}$. Comparing cardinalities, we see that $\mathbb{M}$ has an uncountable transcendence basis $\{x_i\}_{i \in I}$ over $\mathbb{K}$.

Using the second-countability of $\mathbb{R}$ and the fact that $\mathbb{K}$ is a dense subset of $\mathbb{R}$ (with respect to the Euclidean topology), we can assume that $\{x_i\}_{i \in I}$ is dense in $\mathbb{R}$. Let $\{U_n\}_{n \in \mathbb{N}}$ be a basis of the Euclidean topology on $\mathbb{R}$, and let $\{x_n\}_{n \in \mathbb{N}}$ be a countable subset of any transcendence basis $B$ of $\mathbb{M}$ over $\mathbb{K}$. For every $n \in \mathbb{N}$, there exists $q_n \in \mathbb{Q} \subset \mathbb{K}$ such that $x_n + q_n \in U_n$. Replace every $x_n$ in $B$ by $x_n + q_n$, obtaining the set $B'$. Since any nontrivial polynomial equation with coefficients in $\mathbb{K}$ satisfied by some elements of $B'$ gives a nontrivial polynomial equation with coefficients in $\mathbb{K}$ satisfied by some elements of $B$, $B'$ must be a transcendence basis of $\mathbb{M}$ over $\mathbb{K}$. By construction, it is dense in $\mathbb{R}$. 

The following lemma is an immediate consequence of this observation.

\begin{lemma}
    If $(x_1^1, \ldots, x_d^1), \ldots, (x_1^n, \ldots, x_d^n)$ are points in $\mathbb{R}^d$, then for $\epsilon > 0$ there exist $(y_1^1, \ldots, y_d^1), \ldots, (y_1^n, \ldots, y_d^n) \in \mathbb{R}^d$ such that $|x_i^j - y_i^j| < \epsilon$ for every $i, j$, where $y_1^1, \ldots, y_d^1, \ldots, y_1^n, \ldots, y_d^n$ are algebraically independent over the field of real algebraic numbers.
\end{lemma}

In what follows, it might be helpful to think of the numbers $y_1^1, \ldots, y_d^1, \ldots,$ $y_1^n, \ldots, y_d^n$ as variables in the polynomial ring $\mathbb{K}[y_1^1, \ldots, y_d^1, \ldots, y_1^n, \ldots, y_d^n]$ or as the counterpart of the notion of points being in a general position.
\sloppy
\section{Sperner's Lemma, Generalized Y Game and the Generalized \mbox{Shoelace Formula}}

In their original paper McLennan and Tourky \cite{sperner1}, authors used the oriented volume method to prove Sperner's Lemma. For the sake of motivation, we begin our discussion with a proof of a similar result, which is a generalization of the non-draw property of the Y game.

The Y game can be described as follows: on an equilateral triangle subdivided into congruent equilateral subtriangles (hence a polytope with a triangulation), two players take turns coloring vertices of the triangulation (each using exactly one of two distinct colors). When there is a connected monochromatic graph touching all sides of the board, the player who colored it wins. 

There is an obvious generalization to $n$ players and the board being an $n$-simplex (note that the following also applies to fewer than $n$ players and is obviously not true for more than $n$ players): each player, in turn, marks an unmarked vertex of a triangulation of $\Delta^n$, with the objective of being the first player to connect all faces of maximal dimension by a monochromatic subgraph. To show that a draw is impossible in that case, it is enough to show that every coloring of a triangulated $n$-simplex with $n$ colors contains a monochromatic connected subgraph touching every face of the simplex. Usually, the rules of the Y game specify a particular triangulation, but as the reasoning below shows, the property holds for any choice of triangulation.

\begin{proof}[Proof of Theorem \ref{gra-y}]
Suppose that the statement of the theorem is not true. Then there exists a coloring in which each maximal monochromatic connected subgraph $G$ can be assigned a face $F_G$ of the simplex such that $G \cap F_G = \emptyset$. Denote the vertex of $\Delta^n$ opposite to $F_G$ by $O_G$ (see Figure \ref{jeden}). Let each vertex from each $G$ travel in a straight line at a constant speed towards $O_G$ in time $t \in [0,1]$. Consider the sum of volumes of all simplices from the triangulation. It follows that this is a nonzero constant function (note that the $n+1$ vertices of $\Delta^n$ are stationary). However, each simplex from the triangulation has at least two vertices of the same color (and obviously, they are both contained in the same maximal monochromatic connected subgraph), hence they travel to the same vertex of $\Delta^n$. Thus, at time $t=1$, every simplex collapses and its volume is reduced to 0, a contradiction.
\end{proof}

\begin{figure}[H]
    \centering
    \includegraphics[scale=2.0]{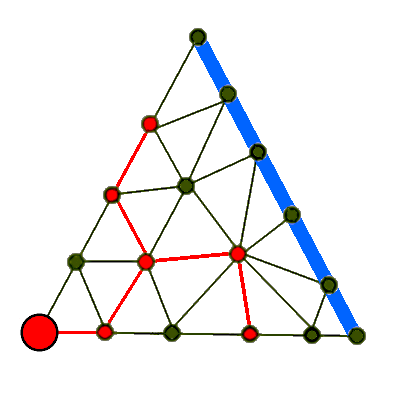}
    \caption{In the picture above, the graph $G$ is depicted in red, the face $F_G$ in blue, and the vertex $O_G$ is indicated by the biggest red dot.}
    \label{jeden}
\end{figure}

For future reference, note that in the proof above there are no restrictions on the geometric realization of a simplex. We are proving a combinatorial theorem using geometric tools; however, no assumptions on the geometric realization are placed. This will not be the case in the following sections.

For the sake of completeness, we describe Sperner's Lemma and roughly sketch its proof as given by McLennan and Tourky.

\begin{theorem}
    Let $\Delta^n$ be an $n$-dimensional simplex with a given triangulation $T$. Let $V$ be the set of vertices of $T$. Assume $V$ is colored with $n+1$ colors in such a way that every extreme point of $\Delta^n$ gets a different label, and every vertex contained in a face spanned by some subset $S$ of the extreme points is colored by a label assigned to some element of $S$. Then there exists an odd number of $n$-simplices $\sigma$ in $T$ such that vertices of $\sigma$ are colored with $n+1$ distinct colors.
\end{theorem}

\begin{proof}
    We can assume that $\textup{Vol}(\Delta^n)=1$. Let every vertex $v$ of $T$ travel at a constant speed in a straight line towards the extreme point of $\Delta^n$ labeled by the same color as $v$. As in the previous proof, we deduce that the sum of volumes of $n$-simplices from $T$ induces a nonzero constant function. At the end of the homotopy, every simplex of $H_1(\sigma)$ gives a summand equal to 0, 1, or $-1$ (respectively: if it has two vertices labeled in the same way, or if it has distinct labels agreeing or not with the labeling of the vertices of $\Delta^n$). Thus, we deduce that there is exactly one more $n$-simplex contributing $1$ than those contributing $-1$ (this is an easy case of the odd-covering theorem discussed later).
\end{proof}
We finish this section with a similar proof of a result that will be used later:
\begin{lemma}[Generalized Shoelace Formula]\label{shoelace}
    If $P$ is a $d$-dimensional polytopal complex in $\mathbb{R}^d$, then $\textup{Vol}(P) = \frac{1}{d!} \sum_{i \in I} \kappa(\sigma_i)$, where $\{\sigma_i\}_{i \in I}$ is the set of faces of $P$ and $\kappa(\sigma_i)$ is the determinant of the matrix with columns given by the vertices of $\sigma_i$ (order of columns is chosen accordingly to the orientation of $\sigma_i$).
\end{lemma}
\begin{proof}
    Triangulate $P$ (without adding any new vertices). Add a new vertex $v_j$ in the interior of every $d$-simplex $\tau_j$ of the triangulation. Divide every $\tau_j$ into $d+1$ $d$-simplices given by the convex hulls of the faces of $\tau_j$ and vertex $v_j$. Let each $v_j$ travel linearly to $0 \in \mathbb{R}^d$ at a constant speed, keeping all other vertices fixed, and consider the homotopy induced by these maps. For small values of $t$, the images of $\tau_j$'s still form a triangulation of $P$, so the sum of their volumes is independent of $t$. Exploiting the expression corresponding to $t=1$, namely the sum of oriented volumes of $d$-simplices with one vertex in $0 \in \mathbb{R}^d$, we see that $\kappa(\sigma_i)$'s corresponding to a face of more than one $\tau_j$ must vanish (each of them appears once with a positive sign and once with a negative sign, since they correspond to a common face of two $d$-simplices with the same orientation). Each of the remaining $\kappa(\sigma_i)$'s appears in the expression exactly once, with the standard orientation induced on the boundary from the $d$-simplex.
\end{proof}
\section{Algebraically Independent Coordinates}
The purpose of this section is mostly motivational and, with the exception of the last theorem, can be skipped by readers interested only in technical results. 

We will now be working with general $d$-polytopal complexes, not necessarily a simplex. As before, the volume-sum corresponding to the appropriate homotopy (in the pictures indicated by the coloring of vertices) is constant, so simplices that survive sum up (taking into account the orientation) to the initial volume. However, a priori we can't deduce how they are placed in the geometric realization, particularly if the images of simplices from the original triangulation via $H_1$ form a cover of the polytopal complex. 

For example, if we consider the triangulation of a polytope formed from two congruent triangles on the plane sharing exactly one edge, as depicted below, by counting volumes we only know that in the end we get at least two simplices with volumes equal to the volume of any of the triangles, and this is not enough to deduce that they are mapped to distinct triangles. However, if the starting blue and yellow triangles are chosen in such a way that the volume of the first one is not a rational multiple of the volume of the second one (in particular, they are not congruent), we can identify the simplices to which $C$ and $D$ must be mapped. To be precise, we need to take into account all the simplices arising from splitting our tetragon by the vertical diagonal, but as we will see later it is possible to find such a geometric realization that the volumes uniquely determine each simplex, no matter from which triangulation it came. Since the coloring is specified by the combinatorial data and not by a particular geometric realization, the claims we make will hold true for every realization. In the left picture below, we see that the triangle $C$ is mapped by the homotopy to the yellow triangle, while the triangle $D$ is mapped to the blue triangle. In the right picture, we can infer that only by comparing the areas: the blue and yellow triangles have different areas, so just by comparing the areas we know that two different simplices must be mapped onto them.

\begin{figure}[H]
    \centering
    \begin{subfigure}[t]{0.36\textwidth}
        \includegraphics[width=\textwidth]{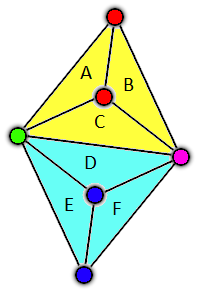}
        \caption{In this realization, it is impossible to distinguish between the yellow and blue triangles judging by their volume.}
        \label{fig:P}
    \end{subfigure}\hspace{10pt}
    \begin{subfigure}[t]{0.36\textwidth}
        \includegraphics[width=\textwidth]{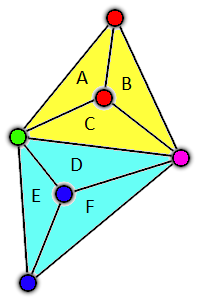}
        \caption{In this realization, that problem is not present.}
        \label{fig:Q}
    \end{subfigure}
    \caption{A comparison between two geometric realizations}
    \label{fig:animals}
\end{figure}

This suggests that it might be profitable to search for geometric realizations with an additional property: that distinguish simplices by their volumes. To that end, one (perhaps the most natural) strategy would be to provide some constructions concerning measure-theoretic properties of sets of solutions to the linear Diophantine equations describing sums of volumes. However, we choose a different approach, relying on field theory. This is simpler and, in some sense, gives the optimal solution to the problem: volumes of simplices and their faces are as linearly independent as possible. It also enables us to use the generalized shoelace formula in a simple manner and demonstrates the connection between the coloring problems and the statements about polynomials.

As one can easily see, any geometric realization of a triangulated polytopal complex remains a geometric realization under sufficiently small perturbation of any of its vertices (note that this is not the case when the polytopal complex is not triangulated, as the perturbation may violate affine dependence of the vertices). Thus we can find a geometric realization such that each coordinate of each vertex is a different element of the set $\{x_i\}_{i \in I}$, the transcendence basis of a field extension $\mathbb{K} \subset \mathbb{R}$ dense in $\mathbb{R}$ constructed at the end of the first section. If the dimension of the ambient space is equal to the dimension of the polytopal complex, we can use the generalized shoelace formula to prove linear independence of those volumes.

Suppose, for the sake of argument, that there exists a nontrivial linear relationship between the volumes of the simplices of some fixed triangulation $T$ with algebraically independent coordinates $\sum_{i \in I} a_i \textup{Vol}(i)$, where $I$ is the set of $d$-dimensional $T$-simplices. By induction, we can assume that for every polytopal complex with a triangulation having a strictly smaller number of $d$-simplices, any geometric realization with algebraically independent coordinates induces linearly independent volumes of the simplices. A monomial corresponding (via the generalized shoelace formula) to a face of any polytope contained in the boundary of that polytopal complex can be produced only from the volume of one specific simplex of the triangulation, namely the one having that face. Let us fix one such simplex and denote it by $j$. Since $\sum_i a_i \textup{Vol}(i) = 0$, we deduce that $a_j = 0$. By the inductive assumption applied to the polytopal complex formed by $I \setminus \{j\}$, all $a_i$'s must be equal to $0$.

Thus our goal is established - there are no nontrivial linear relationships between the volumes of the simplices, where by a trivial relationship we understand those arising from the fact that two sets of simplices may form two triangulations of the same subset of a polytopal complex (this idea will be made precise in the next theorem).

\begin{remark}
    Note that also the set of determinants of the form
    $$\begin{vmatrix} x_1^1 & \ldots & x_n^1 \\ \vdots & \ddots & \vdots \\ x_1^n & \ldots & x_n^n \end{vmatrix}$$
    (that is, determinants denoted by $\kappa(\sigma_i)$ in the generalized shoelace formula) where $(x_i^1, \ldots, x_i^n)$ are points with algebraically independent coordinates, is linearly independent.
\end{remark}

The next theorem will be used in the proof of the generalized Atanassov conjecture in the following section.

\begin{theorem}[Odd-Covering Theorem]
    Let $P$ be a $d$-polytopal complex with a triangulation $T$. For every $T$-vertex $v$, let $A_v$ be the affine space spanned by the minimal cell containing $v$. Assume that for every such $v$, there is a specified curve $\gamma_v: [0,1] \to A_v$ such that $\gamma_v(0) = v$. Let $H_t$ be a homotopy moving every vertex $v$ of $T$ by $\gamma_v$ extended to $P$ simplex by simplex linearly in the barycentric coordinates of $T$. If we denote the $d$-simplices of $T$ by $\tau_1, \ldots, \tau_m$, then up to a measure-zero set, every point of $P$ is contained in an odd number of sets $H_1(\tau_1), \ldots, H_1(\tau_m)$, and there is exactly one more such simplex with positive orientation.
\end{theorem}

\begin{proof}
    It is enough to prove the theorem for a $d$-polytopal complex $P$ embedded in $\mathbb{R}^d$, as we can split $P$ into parts of positive measure that are contained in distinct essential affine $d$-spaces and the homotopy will be independent of each subspace. If $P \subset \mathbb{R}^d$, there exists a $d$-simplex $\sigma$ such that $\textup{int}\ \sigma \supset P$. Extend the triangulation $T$ on $P$ to a triangulation $T'$ of $\sigma$. Define a homotopy $H': \sigma \times [0,1] \to \sigma$ by $H$ on the vertices of $P$, and fixing the vertices of $\sigma$ (and again extending in barycentric coordinates). It is easy to see that $\textup{Vol}\big(H(\sigma \setminus P, t) \cap P\big) = 0$ for any $t \in [0,1]$. Now the claim follows by the Brouwer degree - it must be $1$, since the constructed homotopy starts with an identity, and on the other hand it is computed by the sum of orientations of preimages of any point with a finite fiber.
\end{proof}

The Brouwer degree makes the proof really short and direct. Another way to prove the theorem is to introduce some further subdivisions (by consecutive subdivisions, so that the overlaps of the simplices at time 1 can be interpreted as sums of distinct simplices of some triangulation) in a fixed essential affine space and realize $P$ in this essential space with algebraically independent coordinates. By exploiting the linear independence of the volumes, each of the simplices must appear exactly one more time with a positive orientation than with a negative orientation to cancel out and not introduce a nontrivial relationship.

\section{Atanassov Conjecture}
The Atanassov conjecture (first proved by De Loera, Peterson, and Su in \cite{Atanassov}) is a generalization of Sperner's lemma to the case of polytopes. The labeling of the polytopal complex $P$ with triangulation $T$ is called a \emph{Sperner labeling} if every $T$-vertex $v$ that lies in some $r$-cell $C_i$ is assigned one of the labels of the vertices of $C_i$. This condition assures that the homotopy we will construct will indeed preserve the volume.

\begin{theorem}[Atanassov Conjecture]
    Suppose that $P$ is a (convex) connected $d$-polytope with $n$ vertices labeled with distinct colors and $T$ is a triangulation of $P$ endowed with a Sperner labeling extending the labeling of $P$. Then there are at least $n-d$ $d$-simplices of $T$ such that each of them is labeled with $d+1$ distinct colors.
\end{theorem}

Below we show an easy new proof of a slight generalization of this conjecture, obtained by combining the method of oriented volume and the use of algebraically independent coordinates.

First, we need to introduce a new notion. We say that two $d$-simplices $\sigma$, $\sigma'$ of $P$ are in the same \emph{strongly connected} component if there exists a sequence of $d$-simplices $\sigma_0, \sigma_1, \ldots, \sigma_n$ such that $\sigma_0 = \sigma$, $\sigma_n = \sigma'$ and each pair $\sigma_i, \sigma_{i+1}$ share a common $(d-1)$-face (see \cite{ziegler} p. 101 for a different variant of that notion). It is easy to see that this is an equivalence relation on the set of $d$-simplices of $P$, hence we can define the strongly connected component of some $d$-simplex $\sigma$ as the union of all simplices in its equivalence class. Of course, the triangulation of $P$ restricts to a triangulation of any strongly connected component. We say that a $d$-polytopal complex is strongly connected if it has exactly one strongly connected component.

\begin{theorem}[Generalized Atanassov Conjecture]
    Suppose that $P$ is a strongly connected $d$-polytopal complex which has $n$ vertices on the boundary and $T$ is a triangulation of $P$ endowed with a Sperner labeling extending the labeling of $P$. Then there are at least $n-d$ $d$-simplices of $T$ such that each of them is labeled with $d+1$ distinct colors.
\end{theorem}

\begin{proof}
    Divide $P$ into polytopal complexes $P_1, \ldots, P_m$ that are intersections of $P$ with essential affine $d$-spaces (a point in a $d$-cell lies in $P_i$ iff the whole cell does so) and subdivide each $P_i$ further into its strongly connected components $P_i^1, P_i^2, \ldots, P_i^n$. It is enough to prove the theorem for each such $P_i^j$ (with $k_i^j$ being the number of vertices on its boundary) separately and to use the assumption that $P$ is strongly connected to globalize the result: let $P$ be a strongly connected $d$-polytopal complex that can be divided into two $d$-polytopal complexes $P_1$, $P_2$ satisfying both the assumptions and the statement of our theorem, with $P_1 \cap P_2$ being a set of their common faces of a dimension smaller than $d$. Let $n_i$ be the number of vertices on the boundary of $P_i$ and $s_i$ be the number of its fully colored $d$-simplices. Then $n_i - d \leq s_i$ and since $P$ is strongly connected, $P_1$ and $P_2$ share at least $d$ vertices, so $n - d \leq n_1 + n_2 - 2d \leq s_1 + s_2$.

    Note that the Sperner labeling forces the homotopy to be volume-preserving and to satisfy the odd-covering theorem. Having established that, fix now a geometric realization of $P_i$ in $\mathbb{R}^d$ with algebraically independent coordinates. The odd-covering theorem implies that the images by the homotopy of $d$-simplices in this realization still cover $P_i$; in particular, their sum contains all vertices of $P_i$. First, such image $\tau_1$ contains $d+1$ vertices. If $k_i^j = d+1$, we are done. If not, then there exists at least one face $\sigma$ of $\tau_1$ that is not contained in the boundary of $P_i^j$, hence the determinant corresponding to this face in the expression for $\textup{Vol}(\tau_1)$ must be canceled by the determinant corresponding to some $\textup{Vol}(\tau_2)$. If $k_i^j = d+2$, we are done. If not, then the expression for the volume sum must contain a determinant corresponding to the face of some $\tau_3$ distinct from $\tau_1$ and $\tau_2$ (the formula for the volume must take into account the vertex of $\partial P_i^j$ that is not on the boundary of $\tau_1 \cup \tau_2$). Proceeding this way, we finish when the expression for the volume of $\tau_1 \cup \ldots \cup \tau_l$ agrees with the expression for the volume of $P_i^j$. If there are $\overline{k_i^j}$ vertices of $P_i^j$ that are vertices of some $\tau_i$, then $l \geq \overline{k_i^j} - d$.  
\end{proof}

Below we show a different result obtained by similar reasoning (the homotopy is implicit in the labelings, just as before).

\begin{theorem}
    Let $P$ be a $d$-dimensional polytopal complex. Let $T$ be a triangulation of $P$ endowed with a Sperner labeling, with each vertex on the boundary of $P$ being labeled with a different color. Let $S$ be a subset of simplices of $T$ such that $\partial P \not\subset \bigcup_{s \in S} \partial H_1(s)$. Then there exists a fully-colored simplex in $T$ with a labeling different from the labeling of any simplex $s \in S$.
\end{theorem}

\begin{proof}
    Fix a geometric realization of $P$ with algebraically independent coordinates. Suppose that there is no fully-colored simplex colored differently than simplices of $S$. The sum of volumes induced by the appropriate homotopy remains constant and equal to $\textup{Vol}(P)$. By our assumptions, every simplex labeled differently than simplices in $S$ collapses under the homotopy (since at least two of its vertices are mapped to the same point). Hence, we get that $\textup{Vol}(P)$ is equal to $\sum_{s \in S} n_s \textup{Vol}(s)$ with $n_s \in \mathbb{Z}$. However, $\textup{Vol}(P)$ is linearly independent from $\{\textup{Vol}(s_i)\}_{s \in S}$, since the expression for $\textup{Vol}(P)$ obtained from the generalized shoelace formula contains monomials corresponding to every facet of $P$, and $S$ must be lacking some of them, a contradiction.
\end{proof}

\section{Ky Fan's Lemma}
Our next step is to prove Ky Fan's lemma, again using the volume method and algebraically independent coordinates. The proof below is similar to the original proof by Ky Fan and can be considered as a more algebraic reformulation, showing that the pairing argument used in Fan \cite{kyfan} can be understood on the level of algebraic expressions associated with the volume.

If $P$ is a centrally symmetric polytope homeomorphic with $\mathbb{B}^n$, with $T$ being the triangulation of $P$ which restricts to a centrally symmetric triangulation on $\partial P$, we call the labeling of the vertices of $T$ a \emph{Fan labeling} if:
\begin{itemize}
    \item labels are from $\{-m,-m+1,\ldots,-2,-1,1,2,\ldots,m-1, m\}$ for some $m \in \mathbb{N}$;
    \item opposite vertices on $\partial P$ are labeled by opposite numbers;
    \item there is no edge in $T$ with vertices labeled by opposite numbers.
\end{itemize}

\begin{theorem}[Ky Fan's Lemma, \cite{kyfan}]
    Let $P$ be a barycentric derived subdivision of the octahedral subdivision of the $n$-disk $\mathbb{B}^n$, with $T$ being a triangulation of $P$ which restricts to a centrally symmetric triangulation on $\partial P$. Suppose that $T$ is endowed with a Fan labeling. Then there is an odd number of alternating $n$-simplices in $T$, where a simplex is alternating if it is colored by $n+1$ distinct colors in such a way that when they are ordered by absolute value, their signs alternate.
\end{theorem}

\begin{remark}
    As explained in \cite{kyfan}, "barycentric derived" means "derived by the successive application of a finite number of barycentric subdivisions." Ky Fan's lemma (and its generalized version Theorem \ref{mult-kyfan}) are true for a wider class of centrally symmetric polytopes homeomorphic with $\mathbb{B}^n$, the crucial condition being the possibility to divide their boundary into two antipodally isometric polytopes to which Ky Fan's lemma can be applied. Nonetheless, this is the most common formulation, sufficient for usual topological applications.
\end{remark}

\begin{proof}
    We proceed by induction. The case of $n=1$ is obvious. For the inductive step, pick a realization of $n+1$-dimensional $P$ with algebraically independent coordinates. Note that we will not concern ourselves with preserving the Fan labeling in the geometric sense (as we did with the Sperner labeling before), as the geometric sense of antipodal vertices is irrelevant - we only keep in mind the abstract combinatorial $\mathbb{Z}_2$ action originating from the geometric symmetry.

    Split $\partial P$ into two triangulated hemispheres, which are antipodally symmetric (again, in the abstract sense). We can apply the inductive assumption to them, obtaining that on each there is an odd number of alternating facets and thus the total number of alternating facets of $\partial P$ is divisible by 2, but not by 4, with each facet being paired with its antipodal version. Therefore, there is an odd number of alternating facets whose label of lowest absolute value has a negative sign, and similarly for a positive sign. We call this sign the sign of the alternating facet.

    Apply a homotopy moving each vertex from the interior of $P$ to some vertex of $\partial P$ with the same color (if there is no $i$-colored vertex in $\partial P$, $i$-vertices remain stationary). As always, the volume is preserved under this homotopy, and we inspect the expressions arising from the generalized shoelace formula for each $n+1$-simplex $H_1(\sigma)$. An alternating simplex has two alternating facets - obtained by deleting the vertex with the lowest and highest absolute value of the label, respectively. A non-alternating simplex with distinct absolute values of labels presents two possibilities - it either can have only one pair of adjacent (i.e., when ordered by absolute value) labels of the same sign and thus has two alternating facets obtained by deleting any one of the corresponding vertices (the facets then have the same sign) or has no alternating facets at all. The same holds for simplices with repetitions in labeling.

    By the generalized shoelace formula, the volume of $P$ equals the sum of determinants assigned to facets of the boundary and also equals the sum of $\textup{Vol}(H_t(\sigma))$. Comparing these two expressions and using the facts above, we get the claim.
\end{proof}

\section{Product of Two Simplices and Multilabeled Sperner's Lemma}
In this section, we prove a new theorem about the coloring of the product of two simplices. As a corollary, we obtain the already known multilabeled version of Sperner's lemma. For a purely combinatorial proof and a more elaborate discussion, one can see Meunier and Su \cite{multi}, however, one should note the subtle differences between statements given in this chapter and theorems in Meunier and Su \cite{multi} - we are dealing with triangulated balls instead of arbitrary free $\mathbb{Z}_2$-complexes and our results claim an odd number of simplices of a given type (instead of merely their existence). To this end, we slightly modify the technique of algebraically independent coordinates.

Consider two simplices $\Delta^n$ and $\Delta^m$ embedded in $\mathbb{R}^n$ and $\mathbb{R}^m$. Taking them to be the standard simplices, we get that their volumes are $\frac{1}{n!}$ and $\frac{1}{m!}$. In their product $\Delta^n \times \Delta^m$, consider the fibers of $\pi_2: \Delta^n \times \Delta^m \to \Delta^m$ to be the copies of $\Delta^n$. We can enumerate the vertices of $\Delta^m$ by $v_0, \ldots, v_m$ and in an affine copy of $\mathbb{R}^n$ we can rescale the vertices of the 0-th copy of $\Delta^n$ by the constant $r_0$, and so on with constants $r_0, \ldots, r_m$. The convex hull of all the vertices obtained this way is a new geometric realization of $\Delta^n \times \Delta^m$ in $\mathbb{R}^n \times \mathbb{R}^m$. We call an $m+n$-simplex from the triangulation of $\Delta^n \times \Delta^m$ a $(k_0, \ldots, k_m)$-simplex if it has $k_i$ vertices in $\pi_2^{-1}(v_i)$ for every $i \in \{0, 1, \ldots, m\}$. We define a $(l_0, \ldots, l_n)$-simplex similarly by turning to the other projection.

\begin{lemma}
    The volume of a $(k_0, \ldots, k_m)$-simplex with $k_0 + \ldots + k_m = m+n+1$ is equal to $r_0^{k_0-1} \cdots r_m^{k_m-1} \frac{1}{n! m!} / \binom{n+m}{n}$.
\end{lemma}

\begin{proof}
    It is well known that in the case $r_0 = r_1 = \ldots = r_m = 1$ all the simplices of $\Delta^n \times \Delta^m$ have the same volume and each triangulation contains $\binom{n+m}{n}$ of them (proof can be found in De Loera, Rambau and Santos \cite{triangulations}). Hence, it is enough to show that changing $r_i$ to $r_i'$ corresponds to multiplying the volume by $\big(\frac{r_i'}{r_i}\big)^{k_i-1}$.

    Fix a $(k_0, \ldots, k_m)$-simplex $\sigma$. Consider the affine subspace $W$ of codimension 1 spanned by the vertices of $\Delta^n \times \Delta^m$ that are not in $\pi_2^{-1}(v_i)$. Let $H$ be the height of $\Delta^m$ from the vertex $v_i$ to the opposite face. We foliate $\Delta^n \times \Delta^m$ by the affine spaces parallel to $W$ to compute the volume of $\sigma$ by integrating over $h$ volumes of intersections of $\sigma$ with leaves of foliation as follows. 

    Let $H$ be parametrized linearly by $[0, 1]$. Let $A$ be the image of the set $\sigma \cap W$ via the orthogonal projection onto the subspace which is orthogonal to the affine span of $\sigma \cap \pi_2^{-1}(v_i)$ (we can translate both of those sets to the leaf containing 0). Then the volume of the leaf of $\sigma$ over $t \in [0, 1]$ is given by $(1-t) \ \textup{Vol}(A) \cdot\ t \ \textup{Vol}(\sigma \cap \pi_2^{-1}(v_i))$. Since $\textup{Vol}(\sigma \cap \pi_2^{-1}(v_i))$ is a scalar multiple of $r_i^{k_i-1}$, the claim follows by the Fubini theorem.
\end{proof}

\begin{corollary}
    If $r_0, \ldots, r_m$ are algebraically independent, then the set of volumes of different simplices of a triangulation of $\Delta^n \times \Delta^m$ is linearly independent.
\end{corollary}

We cite again De Loera, Rambau and Santos \cite{triangulations} for the well-known fact that every triangulation of $\Delta^n \times \Delta^m$ (on the vertices of $\Delta^n \times \Delta^m$) contains exactly one $(k_0, \ldots, k_m)$-simplex for every choice of $(k_0, \ldots, k_m)$. In fact, all we need to know is that there exists at least one triangulation with that property - then from the reasoning below follows that this is true for every triangulation.
\begin{proof}[Proof of Theorem \ref{prod-sympleks}]
    Choose $r_0, \ldots, r_m$ in such a way that they form the transcendence basis of $\mathbb{K}(r_0, \ldots, r_m)$ over $\mathbb{K} = \mathbb{Q}^{alg} \cap \mathbb{R}$. Applying the standard homotopy and considering the volume sum, we get the claim.
\end{proof}

\begin{theorem}[Multilabeled Sperner's Lemma]
    Let $T$ be a triangulation of $\Delta^{n-1} = \text{conv}(\{p_1, \ldots, p_n\})$ and let $\lambda_1, \ldots, \lambda_m$ be Sperner labelings on $T$. Then
    \begin{enumerate}
        \item for any choice of positive integers $k_1, \ldots, k_m$ such that $k_1 + \ldots + k_m = m+n-1$, there exists a simplex $\sigma \in T$ on which, for each $i$, the labeling $\lambda_i$ uses at least $k_i$ distinct labels;
        \item for any choice of positive integers $l_1, \ldots, l_n$ such that $l_1 + \ldots + l_n = m+n-1$, there exists a simplex $\tau \in T$ on which, for each $j$, the label $j$ is used in at least $l_j$ labelings.
    \end{enumerate}
\end{theorem}

\begin{proof}
    We are going to treat the labelings as vertices of a simplex $\Delta^{m-1}$. Hence, we have two simplices $\Delta^{n-1}$ and $\Delta^{m-1}$, where the first one is endowed with a triangulation $T$. Take the product of these two simplices. The fibers of $\pi_2: \Delta^{n-1} \times \Delta^{m-1} \to \Delta^{m-1}$ are the copies of $\Delta^{n-1}$, so it makes sense to label $\pi_2^{-1}(\lambda_i)$ with $\lambda_i$. Thus, we have a coloring of the product (subdivided into smaller "level-polytopes" arising from the triangulation $T$). Triangulate this product on its vertices, extending the triangulation $T$ on each copy. We arrange a homotopy as in the proof of Sperner's lemma, with each vertex moving in its $\pi_2$-fiber. The first part of the theorem asserts that there exists a level-polytope which, in its image by the homotopy, contains some $(k_1, \ldots, k_m)$-simplex. The second part asserts that there exists a level-polytope which, in its image by the homotopy, contains some $(l_1, \ldots, l_n)$-simplex. Both claims follow directly from the previous theorem.
\end{proof}

One might be tempted to try to generalize this reasoning to the products of (convex) polytopes or at least the products with a simplex in order to obtain a theorem unifying the Atanassov conjecture with the multilabeled Sperner's lemma. However, this most obvious generalization does not hold, even for a cube (i.e., a product of a square with a simplex).
\begin{figure}[h]
    \centering
    \includegraphics[scale=2]{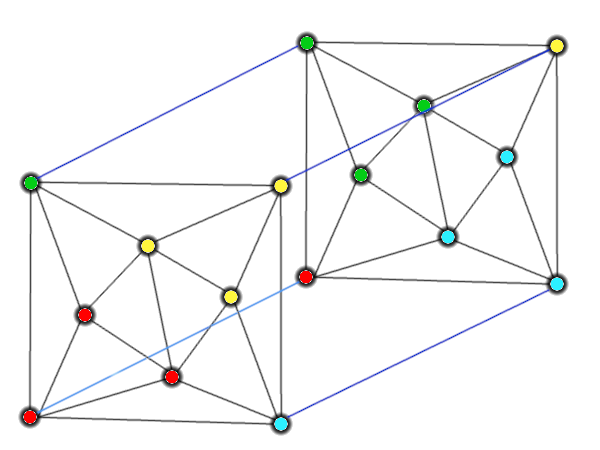}
    \caption{A coloring of vertices assigned to the triangulation of a cube that gives exactly one simplex of the type $(2,2)$.}
\end{figure}

In the picture below, for the sake of clarity, we omitted some of the edges connecting corresponding vertices of two copies of the square. As one can check, this coloring gives rise to only one simplex of the type $(2,2)$ (no matter which subtriangulation we choose, and despite the fact that the colors assigned to that simplex are not uniquely determined), while one would expect two simplices of that type. Again, the sum of the oriented volumes of the simplices of a fixed type does not depend on the choice of the triangulations, but this quantity may not be equally distributed between them, as the example above indicates.

We conclude the paper with the proof of a multilabeled generalization of Ky Fan's lemma stated in the overview.

\begin{proof}[Proof of Theorem \ref{mult-kyfan}]
    The proof is by induction and is similar to the previous proof of the ordinary Ky Fan's lemma, with the aid of the heuristic topological equation:
    $$
        \partial (\mathbb{B}^n \times \Delta^{m-1}) = (\partial \mathbb{B}^n \times \Delta^{m-1}) \cup (\mathbb{B}^n \times \partial \Delta^{m-1}) $$
        $$= (\mathbb{S}^{n-1} \times \Delta^{m-1}) \cup (\mathbb{B}^n \times \partial \Delta^{m-1})
        $$
        $$
        \cong (2 \cdot \mathbb{B}^{n-1} \times \Delta^{m-1}) \cup (\mathbb{B}^n \times m \cdot \Delta^{m-2}).
    $$

    The first step is to prove the theorem in the case $n=1$ (and arbitrary $m$) - in this situation, we are unable to divide the boundary of a ball into two balls of the lower dimension. The case $m=1$ is covered by the ordinary Ky Fan's lemma; therefore, we can proceed by induction on $m$. As before, fix a geometric realization of $P$ with algebraically independent coordinates and take its product with an ordinary realization of $\Delta^m$. Apply the homotopy extended from the fiberwise movement of the vertices (i.e., each vertex in each fiber moves as in the proof of the ordinary Ky Fan's lemma). We inspect parts of $\partial (\mathbb{B}^1 \times \Delta^{m-1})$ isomorphic with $\mathbb{B}^1 \times \Delta^{m-2}$ (preimages of faces of $\Delta^{m-1}$ via $\pi$). By the inductive assumption, on each of them, there is an odd number of alternating simplices of any type: for any sequence $(a_1, \ldots, a_m)$ with $a_i=2, a_j=0$ for some $i, j$ and $a_k=0$ for $k \neq i, j$, there is an odd number of alternating $(a_1, \ldots, a_m)$-simplices in $\partial (\mathbb{B}^1 \times \Delta^{m-1})$ (every simplex contained in $\partial \mathbb{B}^1 \times \Delta^{m-1}$ is of the type $(1, 1, 1, \ldots, 1)$). Every alternating $(a_1, \ldots, a_m)$-simplex is a face of an $(a_1, \ldots, a_j+1, \ldots, a_m)$-simplex (which is trivially seen to also be alternating) and cannot be a face of any nondegenerate simplex of a different type. Furthermore, each $(a_1, \ldots, a_j+1, \ldots, a_m)$-simplex has exactly one such face. Hence, comparing the expressions for the volume of $\mathbb{B}^1 \times \Delta^{m-1}$ obtained by the application of the generalized shoelace formula and by the volume sum at the end of the constructed homotopy, we get the claim.

    We proceed to prove the theorem for any choice of $n$ and $m$, assuming it is true for $\mathbb{B}^{n-1} \times \Delta^{m-1}$ and $\mathbb{B}^n \times \Delta^{m-2}$. As previously, we see that there is an odd number of alternating $(a_1, \ldots, a_m)$-simplices in $\partial (\mathbb{B}^n \times \Delta^{m-1})$. Furthermore, a glimpse at $\partial \mathbb{B}^n \times \Delta^{m-1} \cong (2 \cdot \mathbb{B}^{n-1} \times \Delta^{m-1})$ shows that there is an even (but not divisible by four) number of alternating $(b_1, \ldots, b_m)$-simplices for every $(b_1, \ldots, b_m)$, a sequence of strictly positive natural numbers, just as in the ordinary Ky Fan's lemma. To finish the proof, we express the volume in terms of the faces in the boundary - for the $(a_1, \ldots, a_m)$-simplices we reason as in the $n=1$ case, while for the rest just like in the ordinary Ky Fan's lemma.
\end{proof}
\section{Declarations}
No funding was received to assist with the preparation of this manuscript.
\newpage

\bibliographystyle{siamplain}

\begin{thebibliography}{99}
\bibitem[1978]{fields} J. Browkin, {\it Teoria ciał},  Biblioteka Matematyczna t. 49, PWN, Warszawa 1978  \bibitem[2002]{Atanassov} J. A. De Loera, E. Petersen, F. E. Su, {\it A Polytopal Generalization of Sperner's Lemma},  Journal of Combinatorial Theory, Series A
Volume 100, Issue 1, October 2002, Pages 1-26
\bibitem[2010]{triangulations} J. A. De Loera, J. Rambau, F. Santos, {\it Triangulations: Structures for Algorithms and Applications}, Graduate Texts in Mathematics, Springer-Verlag, Berlin Heidelberg 2010 
\bibitem[1952]{kyfan} Ky Fan, {\it A Generalization of Tucker's Combinatorial Lemma with Topological Applications },  Annals of Mathematics, Second Series, Vol. 56, No. 3 (Nov., 1952), pp. 431-437 
\bibitem[1979]{gale} D. Gale, {\it The Game of Hex and the Brouwer Fixed-Point Theorem}, The American Mathematical Monthly, Vol. 86, No. 10 (1979), pp. 818–827 
\bibitem[2008]{sperner1} A. McLennan, R. Tourky, {\it Using Volume to Prove Sperner's Lemma},  Economic Theory Vol. 35, No. 3 (Jun., 2008), pp. 593-597 
\bibitem[2019]{multi} F. Meunier, F. E. Su, {\it Multilabeled versions of Sperner's and Fan's lemmas and applications},  SIAM J. Appl. Algebra Geometry, 2019, 3(3), 391–411. (21 pages) 

\bibitem[2020]{fields2} J. S. Milne, {\it Fields and Galois theory},  https://www.jmilne.org/math/CourseNotes/FT.pdf, Access~13.08.2020  


\bibitem[2021]{prytula} Tomasz Prytuła, {\it Hex implies Y}, arXiv, 2021, \url{https://arxiv.org/abs/2102.02621}
\bibitem[2012]{ziegler} G.M. Ziegler, {\it Lectures on Polytopes}, Graduate Texts in Mathematics, Springer New York, 2012


\end{thebibliography}

\end{document}